\documentclass[english]{amsart}
\usepackage[T1]{fontenc}
\usepackage[latin9]{inputenc}
\usepackage{mathrsfs}
\usepackage{amsbsy}
\usepackage{amsthm}

\makeatletter
  \theoremstyle{definition}
  \newtheorem{defn}{\protect\definitionname}
\theoremstyle{plain}
\newtheorem{thm}{\protect\theoremname}

\usepackage{upgreek}
\usepackage{amssymb}
\let\originalleft\left
\let\originalright\right
\renewcommand{\left}{\mathopen{}\mathclose\bgroup\originalleft}
\renewcommand{\right}{\aftergroup\egroup\originalright}

\makeatother

\usepackage{babel}
  \providecommand{\definitionname}{Definition}
\providecommand{\theoremname}{Theorem}

\begin{document}

\title[Pi Theorem Revisited]{The Pi Theorem Revisited: \\
On Representations of Quantity Functions}

\author{Dan Jonsson}

\address{Dan Jonsson, University of Gothenburg, Gothenburg, Sweden }

\email{dan.jonsson@gu.se}
\begin{abstract}
This note states and proves a representation theorem for coregular
quantity functions, based on the theory of quantity spaces \cite{key-7},
thereby giving a new perspective on dimensional analysis and the classical
$\pi$ theorem.
\end{abstract}

\maketitle
The central theorem in dimensional analysis is the so-called $\pi$
theorem, with a long history featuring contributions by Fourier \cite{key-4},
Vaschy \cite{key-8}, Federman \cite{key-3}, Buckingham \cite{key-2}
and others. The $\pi$ theorem shows how to represent a ''physically
meaningful'' equation $y=\phi\left(x_{1},\ldots,x_{n}\right)$, describing
a relationship among quantities, as a more informative equation of
the form $y=\prod_{i=1}^{r}\!\xi{}_{i}^{c_{i}}\,\psi\left(\xi_{r+1}',\ldots,\xi_{n}'\right)$,
where $\left(\xi_{1},\ldots,\xi_{n}\right)$ is a permutation of $\left(x_{1},\ldots,x_{n}\right)$,
$c_{i}$ are integers and $\xi_{i}'$ depends on $\xi_{i}$ and $\xi_{1},\ldots,\xi_{r}$.
Following the development of quantity calculus \cite{key-1}, some
quantity calculus versions of the $\pi$ theorem have also been proposed
(see \cite{key-7} for references). This note presents a representation
theorem for a general class of quantity functions, based on the theory
of quantity spaces. Specifically, quantity functions that are ''coregular''
\textendash{} a natural, not too restrictive condition \textendash{}
have useful representations of the form described in Theorem \ref{prop:pi}.

For the sake of completeness, let us briefly review some elements
of the theory of quantity spaces \cite{key-5,key-7}. A \emph{scalable
monoid} over a ring $R$ is a monoid $Q$ equipped with an $R$-action
\[
\cdot:R\times Q\rightarrow Q,\qquad\left(\alpha,x\right)\mapsto\alpha\cdot x,
\]
such that for any $\alpha,\beta\in R$ and $x,y\in Q$ we have (1)
$1\cdot x=x$, (2) $\alpha\cdot\left(\beta\cdot x\right)=\alpha\beta\cdot x$,
and (3) $\alpha\cdot xy=\left(\alpha\cdot x\right)y=x\left(\alpha\cdot y\right)$;
as a consequence, $\left(\alpha\cdot x\right)\left(\beta\cdot y\right)=\alpha\beta\cdot xy$.
We denote the identity element of $Q$ by $1_{\!Q}$, and set $x^{0}=1_{\!Q}$
for any $x\in Q$. An element $x\in Q$ may have an inverse $x^{-1}\in Q$
such that $xx^{-1}=x^{-1}x=1_{\!Q}$.

A (strong) finite \emph{basis} for a scalable monoid $Q$ is a set
$\mathscr{E}=\left\{ e_{1},\ldots,e_{r}\right\} $ of invertible elements
of $Q$ such that every $x\in Q$ has a unique expansion 
\[
x=\mu_{\mathscr{E}}\left(x\right)\cdot\prod_{i=1}^{r}\nolimits e_{i}^{c_{i}},
\]
where $\mu_{\mathscr{E}}\left(x\right)\in R$ and $c_{i}$ are integers.
A finitely generated \emph{quantity space} is a commutative scalable
monoid $Q$ over a field $K$, such that there exists a finite basis
for $Q$. The elements of a quantity space are called \emph{quantities}.
We may think of $\mu_{\mathscr{E}}\left(x\right)$ as the measure
of $x$ relative to $\prod_{i=1}^{r}\nolimits e_{i}^{c_{i}}$, and
indirectly the base units in $\mathscr{E}$.

The relation $\sim$ on $Q$ defined by $x\sim y$ if and only if
$\alpha\cdot x=\beta\cdot y$ for some $\alpha,\beta\in K$ is a congruence
on $Q$. The corresponding equivalence classes are called \emph{dimensions};
$\left[x\right]$ is the dimension that contains $x$. We have $\left[\lambda\cdot x\right]=\left[x\right]$
for any $\lambda\in K$. $Q/{\sim}$ denotes the set of all dimensions
in $Q$; this is a free abelian group with multiplication defined
by $\left[x\right]\left[y\right]=\left[xy\right]$. The identity element
in $Q/{\sim}$ is $\left[1_{Q}\right]$.

In every dimension $\mathsf{C}\in Q/{\sim}$ there is a unique \emph{zero}
\emph{quantity} $0_{\mathsf{C}}$ such that $0\cdot0_{\mathsf{C}}=0_{\mathsf{C}}$
and $0_{\mathsf{C}}\neq1_{Q}$. While $0_{\mathsf{C}}x=0_{\mathsf{C}\left[x\right]}$
for any $x\in Q$, the product of non-zero quantities is a non-zero
quantity. A quantity is invertible if and only if it is non-zero,
and any non-zero quantity $u\in\mathsf{C}$ is a \emph{unit quantity}
for $\mathsf{C}$, meaning that for every $q\in\mathsf{C}$ there
is a unique $\mu\in K$ such that $q=\mu\cdot u$, where $\mu\neq0$
if and only if $q$ is non-zero.

A (dimensional) \emph{quantity function} on $Q$ is a function $\Phi:\mathsf{C}_{1}\times\cdots\times\mathsf{C}_{n}\rightarrow\mathsf{C}_{0}$,
where $\mathsf{C}_{i}\in Q/{\sim}$ for $i=0,1,\ldots,n$. One can
define products and inverses of quantity functions in the usual way. 

We need somewhat more flexible definitions of a covariant scalar representation
and a covariantly representable quantity function than those given
in \cite{key-5}.
\begin{defn}
A \emph{regular} quantity function is a quantity function $\Phi\!:\!\mathsf{C}_{1}\!\times\!\cdots\!\times\!\mathsf{C}_{n}\!\rightarrow\!\mathsf{C}_{0}$
such that there exists, for each $i=0,1,\ldots,n$, a unique tuple
$\left(\mathcal{P}{}_{i1},\ldots,\mathcal{P}_{ir}\right)$ of integers
such that $\mathsf{C}_{i}=\prod_{j=1}^{r}\nolimits\!\mathsf{C}_{j}^{\mathcal{P}_{ij}}$.
A \emph{covariantly representable} regular quantity function on a
quantity space $Q$ over $K$ is a regular quantity function $\Phi$
such that there exists a \emph{covariant scalar representation} of
$\Phi$, that is, a function $\upphi:K^{n}\rightarrow K$ such that
if $E=\left(e_{1},\ldots,e_{r}\right)\in\mathsf{C}_{1}\times\cdots\times\mathsf{C}_{r}$
is a tuple of quantities such that, for $i=0,1,\ldots,n$, each $q_{i}\in\mathsf{C}_{i}$
has a unique expansion 
\[
q_{i}=\mu_{E}\left(q_{i}\right)\cdot\prod_{j=1}^{r}\nolimits e{}_{j}^{\mathcal{P}_{ij}}
\]
then we have
\[
\mu_{E}\left(\Phi\left(q_{1},\ldots,q_{n}\right)\right)=\upphi\left(\mu_{E}\left(q_{1}\right),\ldots,\mu_{E}\left(q_{n}\right)\right)
\]
for any $q_{1},\ldots,q_{n}$. We allow $r=0$ and set $\prod_{j=1}^{0}\mathsf{C}_{j}^{\mathcal{P}_{ij}}=\left[1_{Q}\right]$,
$\prod_{j=1}^{0}e{}_{j}^{\mathcal{P}_{ij}}=1_{Q}$ for $i=0,1,\ldots,n$.
\end{defn}
The idea motivating this definition is that a quantity function $\Phi:\left(q_{1},...,q_{n}\right)\mapsto q_{0}$
can be represented by a scalar function $\upphi:\left(\mu_{E}\left(q_{1}\right),\ldots,\mu_{E}\left(q_{n}\right)\right)\mapsto\mu_{E}\left(q_{0}\right)$
only if $\upphi$ does not depend on $E$, as $\Phi$ does not depend
on $E$ \textendash{} a ''physically meaningful'' equation cannot
depend on the choice of units of measurement. 

For brevity, we can call a covariantly representable regular quantity
function a \emph{coregular} quantity function. Theorem \ref{prop:pi}
concerns coregular quantity functions.
\begin{thm}
\label{prop:pi}Let $Q$ be a finitely generated quantity space over
$K$, and let $\Phi$ be a covariantly representable regular quantity
function
\begin{equation}
\Phi:\mathsf{C}_{1}\times\cdots\times\mathsf{C}_{n}\rightarrow\mathsf{C}_{0},\qquad\left(x_{1},\ldots,x_{r},y_{1},\ldots,y_{n-r}\right)\mapsto y_{0},\label{eq:regular}
\end{equation}
where $\mathsf{C}_{0},\mathsf{C}_{1},\ldots,\mathsf{C}_{n}\in Q/{\sim}$,
such that there exists, for each $i=0,1,\ldots,n$, a unique tuple
$\left(\mathcal{P}{}_{i1},\ldots,\mathcal{P}_{ir}\right)$ of integers
such that 
\begin{equation}
\mathsf{C}_{i}=\prod_{j=1}^{r}\nolimits\!\mathsf{C}_{j}^{\mathcal{P}_{ij}}.\label{eq:cond1}
\end{equation}
Then there exists a unique quantity function of $n-r$ arguments
\[
\Psi:\left[1_{q}\right]\times\cdots\times\left[1_{q}\right]\rightarrow\left[1_{q}\right]
\]
such that if $x_{j}\neq0_{\mathsf{C}_{j}}$ for $j=1,\ldots,r$ then
\begin{gather}
\pi_{0}=\Psi\left(\pi_{1},\ldots,\pi_{n-r}\right),\label{eq:piteorem}
\end{gather}
where $\pi_{0}=y_{0}\left(\prod_{j=1}^{r}x_{j}^{\mathbf{\mathcal{P}}_{0j}}\right)^{-1}$,
$\pi_{k}=y_{k}\left(\prod_{j=1}^{r}x_{j}^{\mathbf{\mathcal{P}}_{\left(k+r\right)j}}\right)^{-1}$
for $k=1,\ldots.,n-r$, or equivalently
\begin{equation}
\Phi\left(x_{1},\ldots,x_{r},y_{1},\ldots,y_{n-r}\right)=\prod_{j=1}^{r}\nolimits\!x_{j}^{\mathcal{P}_{0j}}\,\Psi\left(\pi_{1},\ldots,\pi_{n-r}\right).\label{eq:pitheorem2}
\end{equation}
\end{thm}
\begin{proof}
Set $E=\left(e_{1},\ldots,e_{r}\right)$, where $0_{\mathsf{C}_{j}}\neq e_{j}\in\mathsf{C}_{j}$
for $j=1,\ldots r$. Then we have $\mathsf{C}_{i}=\prod_{j=1}^{r}\mathsf{C}_{j}^{\mathcal{P}_{ij}}=\prod_{j=1}^{r}\left[e_{j}\right]^{\mathcal{P}_{ij}}=\left[\prod_{j=1}^{r}e{}_{j}^{\mathcal{P}_{ij}}\right]$
for $i=0,1,\ldots,n$. $\prod_{j=1}^{r}e{}_{j}^{\mathcal{P}_{ij}}$
is non-zero and hence a unit quantity for $\mathsf{C}_{i}$, so for
$i=0,1,\ldots,n$ and every $q_{i}\in\mathsf{C}_{i}$ there is a unique
$\mu_{E}\left(q_{i}\right)\in K$ and unique integers $\mathcal{P}_{ij}$
such that
\begin{equation}
q_{i}=\mu_{E}\left(q_{i}\right)\cdot\prod_{j=1}^{r}\nolimits e{}_{j}^{\mathcal{P}_{ij}}.\label{eq:exp}
\end{equation}

For $j=1,\ldots,r$, $\mathsf{C}_{j}$ has the unique expansion $\mathsf{C}_{j}=\mathsf{C}_{j}^{1}$
so that $x_{j}=\mu_{E}\left(x_{j}\right)~\cdot~e_{j}$. Thus, we have
$\prod_{j=1}^{r}x{}_{j}^{\mathcal{P}_{ij}}=\prod_{j=1}^{r}\left(\mu_{E}\left(x_{j}\right)\cdot e_{j}\right)^{\mathcal{P}_{ij}}=\prod_{j=1}^{r}\mu_{E}\left(x_{j}\right)^{\mathcal{P}_{ij}}\cdot\prod_{j=1}^{r}e{}_{j}^{\mathcal{P}_{ij}}$
for $i=0,1,\ldots,n$, so $\mu_{E}\left(\prod_{j=1}^{r}x{}_{j}^{\mathcal{P}_{ij}}\right)=\prod_{j=1}^{r}\mu_{E}\left(x_{j}\right)^{\mathcal{P}_{ij}}$. 

It is fairly straightforward to verify that if $q\in\left[1_{Q}\right]$
then $q$ has the unique expansion $q=\mu_{E}\left(q\right)\cdot\prod_{j=1}^{r}\nolimits e{}_{j}^{0}$
relative to any $E$; in particular, $1_{Q}=1\cdot\prod_{j=1}^{r}\nolimits e{}_{j}^{0}$.
This implies that if $q_{i}=\mu_{E}\left(q_{i}\right)\cdot\prod_{j=1}^{r}\nolimits e{}_{j}^{\mathcal{P}_{ij}}$,
where $\mu_{E}\left(q_{i}\right)\neq0$ so that $q_{i}$ is invertible,
then $q_{i}^{-1}$ has the unique expansion $q_{i}^{-1}=\frac{1}{\mu_{E}\left(q_{i}\right)}\cdot\prod_{j=1}^{r}e{}_{j}^{-\mathcal{P}_{ij}}$
relative to $E$. As a further consequence, if $q_{i},r_{i}\in\mathsf{C}_{i}$
and $r_{i}$ is invertible, so that $q_{i}r_{i}^{-1}\in\left[1_{Q}\right]$,
then $q_{i}r_{i}^{-1}$ has the unique expansion $q_{i}r_{i}^{-1}=\mu_{E}\left(q_{i}r_{i}^{-1}\right)\cdot\prod_{j=1}^{r}\nolimits e{}_{j}^{0}$,
where $\mu_{E}\left(q_{i}r_{i}^{-1}\right)=\mu_{E}\left(q_{i}\right)/\mu_{E}\left(r_{i}\right)$,
relative to any $E$.

{\small{}Given $\mathcal{P}_{ij}$ by (}\ref{eq:cond1}{\small{}),
set} $p_{0}=\prod_{j=1}^{r}x_{j}^{\mathcal{P}_{0j}}$, $p_{k}=\prod_{j=1}^{r}x_{j}^{\mathcal{P}_{\left(k+r\right)j}}$
for $k=1,\ldots,n-r$, so that $\left[p_{0}\right]=\left[\prod_{j=1}^{r}\mu_{E}\left(x_{j}\right)^{\mathcal{P}_{0j}}\cdot\prod_{j=1}^{r}e{}_{j}^{\mathcal{P}_{0j}}\right]=\left[\prod_{j=1}^{r}e{}_{j}^{\mathcal{P}_{0j}}\right]=\left[y_{0}\right]$
and $\left[p_{k}\right]=\left[\prod_{j=1}^{r}\mu_{E}\left(x_{j}\right)^{\mathcal{P}_{\left(k+r\right)j}}\cdot\prod_{j=1}^{r}e{}_{j}^{\mathcal{P}_{\left(k+r\right)j}}\right]=\left[\prod_{j=1}^{r}e{}_{j}^{\mathcal{P}_{\left(k+r\right)j}}\right]=\left[y_{k}\right]$.
Also set $\nu_{E}\left(y_{k}\right)=\mu_{E}\left(y_{k}p_{k}^{-1}\right)$
for $k=0,1,\ldots,n-r$ when $p_{k}$ is invertible, so that 
\[
\nu_{E}\left(y_{k}\right)=\mu_{E}\left(y_{k}p_{k}^{-1}\right)=\mu_{E}\left(y_{k}\right)/\mu_{E}\left(p_{k}\right)
\]
relative to any $E$, since $y_{k}p_{k}^{-1}\in\left[1_{Q}\right]$.

It is convenient to denote the sequence of arguments $x_{1},\ldots,x_{r},y_{1},\ldots,y_{n-r}$
by $\boldsymbol{q}$, and the sequence $\mu_{E}\left(x_{1}\right),\ldots,\mu_{E}\left(x_{r}\right),\mu_{E}\left(y_{1}\right),\ldots,\mu_{E}\left(y_{n-r}\right)$
by $\tau_{E}\left(\boldsymbol{q}\right)$. By \linebreak{}
assumption, there is a function $\upphi:K^{n}\rightarrow K$ such
that $\mu_{E}\left(\Phi\left(\boldsymbol{q}\right)\right)=\upphi\left(\tau_{E}\left(\boldsymbol{q}\right)\right)$
for any $\boldsymbol{q}$ and $E$, so as $x_{j}\neq0_{\mathsf{C}_{j}}$
for $j=1,\ldots r$ there is a function $\phi:K^{n}\mapsto K$ such
that for any $\boldsymbol{q}$ and $E$ we have
\begin{gather*}
\mu_{E}\left(\Phi\left(\boldsymbol{q}\right)p_{0}^{-1}\right)=\frac{\mu_{E}\left(\Phi\left(\boldsymbol{q}\right)\right)}{\mu_{E}\left(p_{0}\right)}=\frac{\mu_{E}\left(\Phi\left(\boldsymbol{q}\right)\right)}{\mu_{E}\left(\prod_{j=1}^{r}x_{j}^{\mathcal{P}_{0j}}\right)}=\frac{\upphi\left(\tau_{E}\left(\boldsymbol{q}\right)\right)}{\prod_{j=1}^{r}\mu_{E}\left(x_{j}\right)^{\mathcal{P}_{0j}}}=\phi\left(\tau_{E}\left(\boldsymbol{q}\right)\right).
\end{gather*}

Furthermore, there is, for given $\mathcal{P}_{ij}$, a bijection
between scalar sequences
\begin{gather*}
\omega:\tau_{E}\left(\boldsymbol{q}\right)\:\longmapsto\:\mu_{E}\left(x_{1}\right),\ldots,\mu_{E}\left(x_{r}\right),\frac{\mu_{E}\left(y_{1}\right)}{\prod_{j=1}^{r}\mu_{E}\left(x_{j}\right)^{\mathcal{P}_{\left(1+r\right)j}}},\ldots,\frac{\mu_{E}\left(y_{n-r}\right)}{\prod_{j=1}^{r}\mu_{E}\left(x_{j}\right)^{\mathcal{P}_{nj}}}\:,
\end{gather*}
where $\mu_{E}\left(y_{k}\right)/\prod_{j=1}^{r}\mu_{E}\left(x_{j}\right)^{\mathcal{P}_{\left(k+r\right)j}}=\mu_{E}\left(y_{k}\right)/\mu_{E}\left(p_{k}\right)=\nu_{E}\left(y_{k}\right)$
for $k=1,\ldots,n-r$, so there is a function $\chi=\phi\circ\omega^{-1}:K^{n}\rightarrow K$
such that
\[
\nu_{E}\left(y_{0}\right)=\phi\left(\tau_{E}\left(\boldsymbol{q}\right)\right)=\chi\left(\mu_{E}\left(x_{1}\right),\ldots,\mu_{E}\left(x_{r}\right),\nu_{E}\left(y_{1}\right),\ldots,\nu_{E}\left(y_{n-r}\right)\right).
\]

Now set $X=\left(x_{1},\ldots,x_{r}\right)$. By assumption $x_{j}\neq0_{\mathsf{C}_{j}}$
for $j=1,\ldots,r$, so any $q_{i}\in\mathsf{C}_{i}$ has a unique
expansion of the form (\ref{eq:exp}) relative to both $E$ and $X$,
and $\nu_{E}\left(y_{k}\right)=\nu_{X}\left(y_{k}\right)$ for $k=0,1,\ldots,n-r$
since $\mu_{E}\left(y_{k}p_{k}^{-1}\right)$ does not depend on $E$.

There is thus a function $\psi:K^{n-r}\rightarrow K$ such that
\begin{gather*}
\chi\left(\mu_{E}\left(x_{1}\right),\ldots,\mu_{E}\left(x_{r}\right),\nu_{E}\left(y_{1}\right),\ldots,\nu_{E}\left(y_{n-r}\right)\right)=\nu_{E}\left(y_{0}\right)=\\
\nu_{X}\left(y_{0}\right)=\chi\left(1,\ldots,1,\nu_{X}\left(y_{1}\right),\ldots,\nu_{X}\left(y_{n-r}\right)\right)=\psi\left(\nu_{X}\left(y_{1}\right),\ldots,\nu_{X}\left(y_{n-r}\right)\right)
\end{gather*}
for any $E$, since $x_{j}=1\cdot x_{j}$ so that $\mu_{X}\left(x_{j}\right)=1$
for $j=1,\ldots,r$. 

$1_{Q}$ is a unit quantity for $\left[1_{Q}\right]$, so we can define
a regular quantity function of $n-r$ arguments 
\[
\Psi:\left[1_{q}\right]\times\cdots\times\left[1_{q}\right]\rightarrow\left[1_{q}\right]
\]
 by setting
\[
\Psi\left(\nu_{1}\cdot1_{Q},\ldots,\nu_{n-r}\cdot1_{Q}\right)=\psi\left(\nu_{1},\ldots,\nu_{n-r}\right)\cdot1_{Q}.
\]
Then we have
\[
\nu_{X}\left(y_{0}\right)\cdot1_{Q}=\Psi\left(\nu_{X}\left(y_{1}\right)\cdot1_{Q},\ldots,\nu_{X}\left(y_{n-r}\right)\cdot1_{Q}\right)
\]
since $\nu_{X}\left(y_{0}\right)=\psi\left(\nu_{X}\left(y_{1}\right),\ldots,\nu_{X}\left(y_{n-r}\right)\right)$.

Finally, each $p_{k}$ is non-zero and invertible, so
\begin{gather*}
\nu_{X}\left(y_{k}\right)\cdot1_{Q}=\frac{\mu_{X}\left(y_{k}\right)}{\mu_{X}\left(p_{k}\right)}\cdot p_{k}p_{k}^{-1}=\left(\mu_{X}\left(y_{k}\right)\cdot p_{k}\right)\left(\frac{1}{\mu_{X}\left(p_{k}\right)}\cdot p_{k}^{-1}\right)=\\
\left(\mu_{X}\left(y_{k}\right)\cdot p_{k}\right)\left(\mu_{X}\left(p_{k}\right)\cdot p_{k}\right)^{-1}=y_{k}p_{k}^{-1}
\end{gather*}
for $k=0,1,\ldots,n-r$, since $y_{k}=\mu_{X}\left(y_{k}\right)\cdot p_{k}$
and $p_{k}=\mu_{X}\left(p_{k}\right)\cdot p_{k}$ are expansions of
$y_{k}$ and $p_{k}$ relative to $X$. Thus,
\[
y_{0}p_{0}^{-1}=\Psi\left(y_{1}p_{1}^{-1},\ldots,y_{n-r}p_{n-r}^{-1}\right),
\]
or, using the notation $\pi_{k}=y_{k}p_{k}^{-1}$ as in the statement
of the theorem, 
\[
\pi{}_{0}=\Psi\left(\pi{}_{1},\ldots,\pi{}_{n-r}\right).
\]

We have thus shown the existence of a representation of $\Phi$ of
the form (\ref{eq:piteorem}) or (\ref{eq:pitheorem2}). Also, $p_{0}\Psi\left(\pi{}_{1},\ldots,\pi{}_{n-r}\right)=p_{0}\Psi'\left(\pi{}_{1},\ldots,\pi{}_{n-r}\right)$
implies $\Psi\left(\pi{}_{1},\ldots,\pi{}_{n-r}\right)=\Psi'\left(\pi{}_{1},\ldots,\pi{}_{n-r}\right)$
since $p_{0}$ is invertible, so $\Psi$ is unique.
\end{proof}
In this proof, $E$ is not assumed to be a finite basis $\mathscr{E}$
for $Q$, and the assumption that there exists a finite basis for
$Q$ is used only indirectly. This would seem to facilitate the generalization
of Theorem \ref{prop:pi} to other commutative scalable monoids over
a field than finitely generated quantity spaces. 

The condition that $x_{1},\ldots,x_{r}$ are non-zero quantities is
natural and necessary; note that $y_{0},y_{1},\ldots,y_{n-r}$ are
not restricted. It is often assumed in connection with the $\pi$
theorem that quantities are positive (or have positive measures).
This presupposes that $K$ is an ordered field such as the real numbers,
but the present representation theorem holds for any field, for example,
the complex numbers.

$Q$ has no zero divisors, so (\ref{eq:pitheorem2}) implies that
$\Phi\left(x_{1},\ldots,x_{r},y_{1},\ldots y_{n-r}\right)=0_{\mathsf{C}_{0}}$
if and only if $\Psi\left(\pi_{1},\ldots,\pi_{n-r}\right)=0_{\left[1_{Q}\right]}$,
given that $x_{j}\neq0_{\mathsf{C}_{j}}$ for $j=1,\ldots,r$ so that
$p_{0}\neq0_{\mathsf{C}_{0}}$. This is analogous to the form of the
$\pi$ theorem given by Vaschy \cite{key-8} and Buckingham \cite{key-2},
whereas Federman \cite{key-3} proved an identity of the form (\ref{eq:pitheorem2}).
(In hindsight, these early contributions have flaws, but this was
pioneering work.)

It follows from (\ref{eq:pitheorem2}) that if $\Phi$ is a coregular
quantity function and $\Psi\left(\pi_{1},\ldots,\pi_{n-r}\right)$
is invariant under the transformation $x_{j}\mapsto\lambda\cdot x_{j}$
then
\[
\Phi\left(x_{1},\ldots,\lambda\cdot x_{j},\ldots,x_{r},y_{1},\ldots,y_{n-r}\right)=\lambda^{\mathcal{P}_{0j}}\cdot\Phi\left(x_{1},\ldots,x_{r},y_{1},\ldots,y_{n-r}\right)
\]
for non-zero $\lambda,x_{1},\ldots,x_{r}$. Identities of this form
are scaling laws.

Consider a coregular quantity function 
\[
\Phi:\mathsf{C}_{1}\times\cdots\times\mathsf{C}_{n}\rightarrow\mathsf{C}_{0},\qquad\left(x_{1},\ldots,x_{r},y_{1},\ldots,y_{n-r}\right)\mapsto y_{0}
\]
on $Q$ such that each $\mathsf{C}_{i}$ has a unique expansion $\mathsf{C}_{i}=\prod_{j=1}^{r}\nolimits\!\mathsf{C}_{j}^{\mathcal{P}_{ij}}$.
If $\mathsf{D}_{k}\in Q/{\sim}$, $\mathcal{P}_{k}>0$ and $\mathsf{D}_{k}^{\mathcal{P}_{k}}=\mathsf{\mathsf{C}}_{k}$
for $k=0,1,\ldots,n-r$, so that $\Phi$ can be expressed as
\[
\Phi:\mathsf{C}_{1}\times\cdots\times\mathsf{C}_{r}\times\mathsf{D}_{1}^{\mathcal{P}_{1}}\times\cdots\times\mathsf{D}_{n-r}^{\mathcal{P}_{n-r}}\rightarrow\mathsf{D}_{0}^{\mathcal{P}_{0}}
\]
then $\Phi\left(x_{1},\ldots,x_{r},z_{1}^{\mathcal{P}_{1}},\ldots,z_{n-r}^{\mathcal{P}_{n-r}}\right)=z_{0}^{\mathcal{P}_{0}}$,
where $z_{k}\in\mathsf{D}_{k}$, since $\left[z_{k}\right]^{\mathcal{P}{}_{k}}=\left[z_{k}^{\mathcal{P}{}_{k}}\right]$,
and by Theorem \ref{prop:pi} there exists a unique quantity function
of $n-r$ arguments
\[
\Psi:\left[1_{q}\right]\times\cdots\times\left[1_{q}\right]\rightarrow\left[1_{q}\right]
\]
such that if $x_{j}\neq0_{\mathsf{C}_{j}}$ for $j=1,\ldots,r$ then
\begin{gather}
\pi_{0}=\Psi\left(\pi_{1},\ldots,\pi_{n-r}\right),\label{eq:pitheorem3}
\end{gather}
where $\pi_{0}=z_{0}^{\mathcal{P}_{0}}\left(\prod_{j=1}^{r}x_{j}^{\mathbf{\mathcal{P}}_{0j}}\right)^{\!-1}$,
$\pi_{k}=z_{k}^{\mathcal{P}_{k}}\left(\prod_{j=1}^{r}x_{j}^{\mathbf{\mathcal{P}}_{\left(k+r\right)j}}\right)^{\!-1}$
for $k=1,\ldots.,n-r$, or equivalently
\begin{equation}
\Phi\left(x_{1},\ldots,x_{r},z_{1}^{\mathcal{P}_{1}},\ldots,z_{n-r}^{\mathcal{P}_{n-r}}\right)^{\mathcal{P}{}_{0}}=\prod_{j=1}^{r}\nolimits\!x_{j}^{\mathcal{P}_{0j}}\,\Psi\left(\pi_{1},\ldots,\pi_{n-r}\right).\label{eq:piTheorem-3}
\end{equation}
Relation (\ref{eq:piTheorem-3}) is a more generally expressed and
often more useful form of (\ref{eq:pitheorem2}).

In practice, dimensional analysis usually starts from a dimensional
matrix showing how the dimensions of the quantities to be related
by an expression of the form (\ref{eq:piteorem}), (\ref{eq:pitheorem2}),
(\ref{eq:pitheorem3}) or (\ref{eq:piTheorem-3}) are expressed as
products of powers of dimensions of certain base units. How to prepare
data in the form of a dimensional matrix so that we can apply (\ref{eq:piTheorem-3})
is explained in \cite{key-5}. In addition, several examples of dimensional
analysis from beginning to end are given in \cite{key-5,key-6}.

The application of Theorem \ref{prop:pi} to dimensional analysis
is based on the premise that the quantity function posited is covariantly
representable. This is, in fact, a theoretical assumption about the
equivalence of reference frames. An assumption of this kind, together
with others, can be used to derive ''laws of nature'', much as,
for example, the assumption in classical mechanics of the equivalence
of different Galilean reference frames, the Galilean principle of
relativity.

\end{document}